\documentclass{article}%
\usepackage{graphicx}
\usepackage{amsmath,amssymb,amsfonts}%
\usepackage{theorem}%
\usepackage{color}%
\usepackage{hyperref}
\usepackage[font={small, it}]{caption}
\usepackage{caption}
\usepackage{subcaption}
\usepackage{tikz}

\theoremstyle{change}%
\newtheorem{definition}{Definition:}[section]%
\newtheorem{proposition}[definition]{Proposition:}%
\newtheorem{theorem}[definition]{Theorem:}%
\newtheorem{lemma}[definition]{Lemma:}%
\newtheorem{corollary}[definition]{Corollary:}%
{\theorembodyfont{\rmfamily}\newtheorem{remark}[definition]{Remark:}}%
{\theorembodyfont{\rmfamily}}%

\newenvironment{proof}{{\bf Proof:}}
{\qquad \hspace*{\fill} $\Box$}%

\newcommand{\fg}{\mathfrak{g}}%
\newcommand{\fp}{\mathfrak{p}}%

\newcommand{\ad}{\operatorname{ad}}%
\newcommand{\tr}{\operatorname{tr}}%
\newcommand{\id}{\operatorname{id}}
\newcommand{\inner}{\operatorname{int}}%
\newcommand{\fix}{\operatorname{fix}}%
\newcommand{\rme}{\mathrm{e}}%
\newcommand{\EC}{\mathcal{E}}%
\newcommand{\RC}{\mathcal{R}}%
\newcommand{\AC}{\mathcal{A}}%
\newcommand{\DC}{\mathcal{D}}%
\newcommand{\NC}{\mathcal{N}}%
\newcommand{\HC}{\mathcal{H}}%

\newcommand{\VC}{\mathcal{V}}%

\newcommand{\N}{\mathbb{N}}%
\newcommand{\R}{\mathbb{R}}%
\usepackage[pagewise]{lineno}

\begin{document}

\title{The chain recurrent set of flow of automorphisms on a decomposable Lie group}
\author{Adriano Da Silva\thanks{Supported by Proyecto UTA Mayor Nº 4781-24} \\
	Departamento de Matem\'atica,\\Universidad de Tarapac\'a - Sede Iquique, Chile \and Jhon Eddy Pariapaza Mamani \\
Programa Magister en Ciencias con Menci\'on Matem\'atica\\ Departamento de Matem\'atica,\\Universidad de Tarapac\'a - Arica, Chile.
}
\date{\today}
\maketitle

\begin{abstract}
	In this paper we show that the chain recurrent set of a flow of automorphisms on a connected Lie group coincides with the central subgroup of the flow, if the group is decomposable. Moreover, in the decomposable case, the flow satisfies the restriction property. Furthermore, the restriction of any flow of automorphisms to the connected component of the identity of its central subgroup is chain transitive.  
\end{abstract}

	

\section{Introduction}

The concept of chain recurrence set, introduced in the 1970s by C. Conley \cite{Conley}, has become a cornerstone in the study of dynamical systems, developed in the ensuing decades. This concept, defined by means of chains (or pseudo-orbits), has been successfully generalized in various directions not long after its introduction to analyze the dynamics of a broad variety of dynamical systems.

In the present paper, we focus on the recurrent set $\RC_C(\varphi)$ of a flow of automorphisms $\varphi:=\{\varphi_t\}_{t\in\R}$ on a connected Lie group $G$. By considering the Jordan decomposition of $\varphi$, 
$$\varphi=\varphi^{\HC}\circ\varphi^{\EC}\circ\varphi^{\NC},$$
into its hyperbolic $\varphi^{\HC}$, elliptic $\varphi^{\EC}$, and nilpotent $\varphi^{\NC}$ parts, we are able to show that the recurrence set of $\varphi$ coincides with the central subgroup $G^0$ of $\varphi$ (the set of fixed points of the hyperbolical part $\varphi^{\HC}$ of $\varphi$) when the flow decomposes $G$ (see definition below). In particular, the restriction property 
$$\RC_C(\varphi)=\RC_C\left(\varphi|_{\RC_C(\varphi)}\right),$$
holds under such assumptions. We start by showing that the recurrent set does not depend on the elliptical part, since it is a flow of isometries for some left-invariant metric. In sequence, by using the concept of uniform neighborhood, we are able to show that, in the decomposable case,  $\RC_C(\varphi)$ coincides with the recurrence set of the restriction $\varphi|_{G^0}$, which allows us to reduce our analysis to nilpotent flows of automorphisms on connected groups. By considering induced flows on homogeneous spaces, we are able to project and lift chains (under some topological assumptions), which allows us to show that nilpotent flows of automorphisms on connected groups are in fact chain transitive.

The paper is structured as follows: In Section 2, we introduce the Jordan decomposition of a flow of automorphisms at the group level and of its associated derivation at the algebra level. Such decomposition allows us to introduce the concept of unstable, central, and stable subgroups associated with a flow of automorphisms, which are important due to their strict relation with the dynamics of the flow (see \cite{DSAyGZ, DSAyPH, DS}). In Section 3, we define the concept of chains, the chain recurrent set, and the chain transitive set for the flow of automorphisms. Our definition only depends on the topology of the group; however, we show that it coincides with the standard metric definition of chains by considering left-invariant metrics. We also prove several general results relating the decompositions induced by the flow and its Jordan components with the chain recurrence set, which allow us to simplify the proof of our main result. In Section 4, we show our main result, namely that the recurrent set of a flow of automorphisms that decomposes the group coincides with the central subgroup. The proof is done first by considering nilpotent flows of automorphisms on solvable and semisimple Lie groups separately and using them to show the general case. Some concepts and technical results that are used in the main results are stated in an appendix, Section A.

{\bf Notations:} Let $G$ be a connected Lie group. We denote by $\VC_G$ the set of neighborhoods of the identity element $e\in G$.  For any subgroup $H\subset G$, we denote by $H_0$ the connected component of $H$ that contains the identity element $e\in G$. The left and right translations by an element $g\in G$ are denoted by $L_g$ and $R_g$, respectively. The conjugation by $g$ is the map $C_g:= L_g \circ R_{g^{-1}}$. The center $Z(G)$ of $G$ is the set of points satisfying $C_g=\id_G$. For any differentiable map $f:G\rightarrow H$, its differential at $x\in G$ is denoted by $(df)_x$.

\section{Preliminaries}

This section is devoted to presenting the main background needed to establish the mains results of the paper. 

\subsection{Decompositions at the algebra level}

Let $\DC$ be a linear map of $\fg$. We say that $\DC$ is {\bf nilpotent} if $\DC^n\equiv0$ for some $n\in\N$ and $\DC$ is {\bf elliptic} ({\bf resp. hyperbolic}) if it is semisimple and its eigenvalues are pure imaginary ({\bf resp. real}). The {\bf (additive) Jordan decomposition} of $\DC$ is given by the
$$\DC=\DC_{\HC}+\DC_{\EC}+\DC_{\NC}, \hspace{.5cm}\mbox{ where 
}\hspace{.5cm}\DC, \DC_{\HC}, \DC_{\EC}\mbox{ and }\DC_{\NC}\mbox{ commute, }$$
and $\DC_{\HC}$ is hyperbolic, $\DC_{\EC}$ is elliptic, and $\DC_{\NC}$ is nilpotent. Moreover, if $\DC$ is a derivation, $\DC_{\HC}, \DC_{\EC}$, and $\DC_{\NC}$ are also derivations. For $\lambda\in\R$ let us consider the eigenspace of $\DC_{\HC}$ given by
$$\mathfrak{g}_{\lambda}=\{X\in \mathfrak{g}:\mathcal{D}_{\HC}X=\lambda X\}, \hspace{.5cm}\mbox{ where }\hspace{.5cm}\fg_{\lambda}=\{0\} \mbox{ if }\lambda\mbox{ is not an eigenvalue of }\DC_{\HC}.
$$

Following \cite[Proposition 3.1]{SM1}, the  eigenspaces of $\DC_{\HC}$ satisfy  $[\fg_{\lambda}, \fg_{\mu}]\subset 
\fg_{\lambda +\mu }$ when $\lambda+\mu $ is an eigenvalue of $\mathcal{D}_{\HC}$ and zero otherwise. We define the {\bf unstable, central }and {\bf stable} subalgebras of $\fg$, respectively, by
\begin{equation*}
\mathfrak{g}^{+}=\bigoplus_{\lambda :\, \lambda>0}\mathfrak{g}_{\lambda },\hspace{1cm}\mathfrak{g}^{0}=\ker\DC_{\HC}\hspace{1cm}%
\mbox{ and }\hspace{1cm}\mathfrak{g}^{-}=\bigoplus_{\lambda:\, \lambda<0}\mathfrak{g}_{\lambda}.
\end{equation*}
Since $\DC_{H}$ is diagonalizable, $\fg$ decomposes as the direct sum $\mathfrak{g}=\mathfrak{g}^{+}\oplus \mathfrak{g}^{0}\oplus \mathfrak{g}^{-}$. Moreover, the previous properties imply that $\mathfrak{g}^{+},\mathfrak{g}^{0}$ and $\mathfrak{g}^{-}$ are  $\DC, \DC_{\HC}, \DC_{\EC}, \DC_{\NC}$-invariant Lie subalgebras of $\fg$ with $\mathfrak{g}^{+}$, $\mathfrak{g}^{-}$ nilpotent ones. 

\begin{remark}
\label{oldnotation}
   Since all the derivations $\DC, \DC_{\HC}, \DC_{\NC}, \DC_{\EC}$ commute, the subalgebras $\mathfrak{g}^{+}, \mathfrak{g}^{0}$, and $\mathfrak{g}^{-}$ coincide with the sum of the (generalized) eigenspaces of the derivation $\DC$ associated to eigenvalues with positive, zero, and negative real parts, respectively.
\end{remark}

\subsection{Decompositions at the group level}

Let $G$ be a connected Lie group with Lie algebra $\mathfrak{g}$ identified with the set of right-invariant vector fields on $G$.

A {\bf flow of automorphisms} of $G$ is a $1$-parameter subgroup $\varphi:=\{\varphi_t\}_{t\in\R}$ of $\mathrm{Aut}(G)$. Any flow of automorphism $\varphi$ is associated with a unique derivation $\DC:\fg\rightarrow\fg$ by the relation 
\begin{equation*}
\forall  t\in\R\hspace{.5cm}(d\varphi_{t})_{e}=\mathrm{e}^{t\mathcal{D}}, \label{derivativeonorigin}
\end{equation*}
where $\rme^{t\DC}$ stands for the matrix exponential. We say that a flow of automorphisms $\varphi$ is {\bf hyperbolic, elliptic,} or {\bf nilpotent} if its associated derivation $\DC$ is one of the respective types. By the results in \cite{DSAyPH}, the Jordan decomposition of $\DC$ implies the commutative decomposition of $\varphi$ (also called the Jordan decomposition) as
$$\forall t\in\R, \hspace{.5cm} \varphi_t=\varphi_t^{\HC}\circ\varphi_t^{\EC}\circ\varphi_t^{\NC},$$
where $\varphi^{\HC}, \varphi^{\EC}$, and $\varphi^{\NC}$ are hyperbolic, elliptical, and nilpotent flows of automorphisms, respectively.
Let us define the subgroup $G^{0}:=\fix\left(\varphi^{\HC}\right)$, whose Lie algebra is $\fg^0$. Moreover, let us denote by
$G^{+}$ and $G^{-}$ are the connected subgroups of $G$ whose Lie algebras are given by $\mathfrak{g}^{+}$ and $\mathfrak{g}^{-}$. The fact that $G^0$ normalizes $G^+$ and $G^-$ implies that $G^{\pm, 0}:=G^{\pm}G^0$ are Lie subgroups, whose Lie algebra are given by $\mathfrak{g}^{\pm, 0}:=\mathfrak{g}
^{\pm}\oplus \mathfrak{g}^{0}$. By \cite[Proposition 2.9]{DS} and Remark \ref{oldnotation}, all of the previous subgroups are $\varphi, \varphi^{\HC}, \varphi^{\EC}, \varphi^{\NC}$-invariant, closed, and have trivial intersection, that is,
$$G^+\cap G^-=G^+\cap G^{-, 0}=G^-\cap G^{+, 0}=\ldots=\{e\}.$$
Due to their dynamical importance, the subgroups $G^+, G^0$, and $G^-$ are called the {\bf unstable}, {\bf central}, and {\bf stable} subgroups of $\varphi$, respectively.

We say that $G$ is {\bf decomposable} by $\varphi$ if 
$$G=G^{+, 0}G^-=G^{-, 0}G^+.$$ 

Let us note that, since we are assuming $G$ to be connected, a necessary condition for decomposability is that $G^0$ is also connected. In particular, a semisimple Lie group $G$ is decomposable if and only if $G=G^0$ (see \cite{DSAyGZ}). Following \cite[Proposition 2.9]{DS}, a solvable Lie group $G$ is in fact decomposable by any flow of automorphisms. For general Lie groups, compactness of the central subgroup $G^0$ is a sufficient condition to assure decomposability (see \cite[Proposition 3.3]{DSAyGZ}). An important feature of decomposable groups is that any element can be uniquely decomposed into the product of factors in the stable, central, and unstable subgroups.




\begin{remark}
    \label{metricas}
By \cite[Theorem 2.4]{DSAyPH}, any elliptical flow $\varphi$ is a flow of isometries for some left-invariant metric $\varrho$ on $G$. In particular, there exists $V\in\VC_G$ such that $\varphi_t(V)=V$ for all $t\in\R$. On the other hand, if $\varphi$ is a flow whose associated derivation has only eigenvalues with real parts of the same sign, then for any invariant metric there exist $c, \lambda>0$ such that
$$\forall x, y\in G, t>0\hspace{.5cm}\varrho(\varphi_{\epsilon t}(x), \varphi_{\epsilon t}(y))\leq c\varrho(x, y),$$
where $\epsilon=1$ if the real parts of the eigenvalues of $\DC$ are all negative and $\epsilon=-1$ if they are all negative. In particular, for any compact neighborhood $V\in\VC_G$, there exists $\tau_0>0$ such that $\varphi_{\epsilon\tau}(V)\subset V$ for any $\tau\geq\tau_0$.  These standard remarks holds for instance for the restrictions of $\varphi$ to $G^+$ and $G^-$. Such fact will be useful ahead.
\end{remark}

\section{Chains and the recurrence set}

Let $G$ be a connected Lie group with Lie algebra $\fg$ identified with the set of right-invariant vector fields and consider $\varphi:=\{\varphi_t\}_{t\in\R}$ to be a flow of automorphisms. For $U\in\VC_G$, $\tau>0$ and $x, y\in G$. An {\bf $(U, \tau)$-chain} $\xi$ from $x$ to $y$ is given by a finite set $\xi:=\{n; x_0, \ldots, x_n, \tau_0, \ldots, \tau_{n-1}\}$ where $x_i\in G$, $\tau_i\geq \tau$,  
$$x_0=x, \hspace{.5cm}x_n=y\hspace{.5cm}\mbox{ and }\hspace{.5cm}\forall i=0, \ldots, n-1, \hspace{.5cm}x_{i+1}^{-1}\varphi_{\tau_i}(x_i)\in U.$$
The natural $n$ is said to be the number of {\bf jumps} of $\xi$, and $\tau(\xi):=\tau_0+\tau_1+\cdots\tau_{n-1}$ is the {\bf total time} of the chain $\xi$. If $\xi=\{n; x_0, \ldots, x_n, \tau_0, \ldots, \tau_{n-1}\}$ and $\xi'=\{m; y_0, \ldots, y_m, \tau'_0, \ldots, \tau'_{m-1}\}$ are $(U, \tau)$-chains such that $x_n=y_0$, we define the {\bf concatenation} of $\xi$ and $\xi'$ to be the chain
$$\xi\vee\xi':=\{n+m; x_0, \ldots, x_n=y_0, y_1, \ldots, y_m, \tau_0, \ldots, \tau_{n-1}, \tau'_0, \ldots, \tau'_{m-1}\}.$$
The {\bf chain limit set} of a point $x\in G$ is by definition the set
$$\Omega(x):=\bigcap_{U, \tau}\Omega(x, U, \tau), \hspace{.5cm}\Omega(x, U, \tau)=\{y\in G; \,\exists \,(U, \tau)\mbox{-chain from }x\mbox{ and }y\}.$$
A point $x\in G$ is said to be {\bf chain recurrent} if $x\in\Omega(x)$. The {\bf chain recurrent set } of the flow $\varphi$ is by definition the set
$$\RC_C(\varphi):=\{x\in G; \;x\in\Omega(x)\}.$$
In the set $\RC_C(\varphi)$ one can define the following equivalence relation
$$x\sim y\hspace{.5cm}\iff\hspace{.5cm} x\in\Omega(y)\hspace{.5cm}\mbox{ and }\hspace{.5cm}y\in\Omega(x).$$
The classes of the previous equivalence relation are called the {\bf chain transitive components} of $\RC_C(\varphi)$. A subset $F\subset G$ is said to be {\bf chain transitive} if it is contained in a chain transitive component of $\RC_C(\varphi)$. We say that $\varphi$ is {\bf chain transitive} if $G$ is chain transitive.

\begin{remark}
The advantage of the previous definition of chains is that it only considers the topology on $G$. However, one can recover the standard definition of chains by considering a left-invariant metric, as we show in the sequence:  Let $\varphi$ be a flow of automorphisms on $G$ and $\varrho$ a left-invariant metric on $G$. For any $\varepsilon, \tau>0$, a {\bf $(\varepsilon, \tau)$-metric chain} $\xi$ is given by the set $\xi:=\{n; x_0, \ldots, x_n, \tau_0, \ldots, \tau_{n-1}\},$ satisfying
$$\forall i=0, \ldots, n-1, \hspace{.5cm}\varrho(\varphi_{\tau_i}(x_i), x_{i+1})<\varepsilon.$$

The left-invariance of $\varrho$ implies that 
$$\varrho(\varphi_{\tau_i}(x_i), x_{i+1})<\varepsilon\hspace{.5cm}\iff\hspace{.5cm}\varrho(x_{i+1}^{-1}\varphi_{\tau_i}(x_i), e)<\varepsilon\hspace{.5cm}\iff\hspace{.5cm}x_{i+1}^{-1}\varphi_{\tau_i}(x_i)\in B(e, \varepsilon),$$
which implies that the chain limit and chain recurrent sets defined previously can be recovered by left-invariant metrics. In particular, when working on homogeneous spaces, we will use the standard metric concept.
\end{remark}

The next proposition shows many important properties of the chain recurrent set of a flow of automorphisms $\varphi$. They will be important in proving our main result.

\begin{proposition}
\label{properties}
    Let $\varphi$ be a flow of automorphisms. It holds:
\begin{enumerate}
    \item The chain limit sets are $\varphi$-invariant. Consequently, the same is true for any chain-transitive component of $\RC_C(\varphi)$;

\item If $\psi$ is an elliptical flow  satisfying
$$\varphi_t\circ\psi_s=\psi_s\circ\varphi_t, \hspace{.5cm}\forall t, s\in\R,$$
then 
$$\RC_C(\varphi)=\RC_C(\varphi\circ\psi).$$
In particular, $\RC_C(\varphi)=\RC_C(\varphi^{\HC, \NC})$

    \item Any $(U, \tau)$-chain from $x$ to $y$ with total time $T$ induces $(U, \tau)$-chains $\xi$ from $gx$ to $\varphi_T(g)y$ and $\xi'$ from $\varphi_{-T}(g)x$ to $gy$ with $\tau(\xi)=\tau(\xi')=T$. 

\item $\RC_C(\varphi)=\RC_C(\varphi^*)$, where $\varphi^*$ is the reverse flow of automorphisms defined as $\varphi^*_t=\varphi_{-t}$, $t\in\R$.

\end{enumerate}
\end{proposition}

\begin{proof}
    1. Let $y\in \Omega(x)$, $T\in\R$, $U\in\mathcal{V}_G$, and $\tau>0$. Since $e\in G$ is a fixed point of $\varphi_T$, there exists $V\in\VC_G$ such that $V\cup\varphi_T(V)\subset U$. Let $\tau'\geq\max\{\tau\pm T, \tau\}$ and consider a $(V, \tau')$-chain $\xi$ from $x$ to $y$. If $\xi:=\{n; x_0, \ldots, x_n, \tau_0, \ldots, \tau_{n-1}\}$, then 
    $$ \varphi_{-T}\left(\varphi_T(y)\varphi_{\tau'+T}(x_{n-1})\right)=y^{-1}\varphi_{\tau_{n-1}}(x_{n-1})\in V\hspace{.5cm}\implies\hspace{.5cm}\varphi_T(y)\varphi_{\tau'+T}(x_{n-1})\in \varphi_T(V)\subset U,$$
showing that $\{n; x_0, \ldots, \varphi_T(y), \tau_0, \ldots, \tau_{n-1}+T\}$ is an $(U, \tau)$-chain from $x$ to $\varphi_T(y)$, and hence, $\varphi_T(\Omega(x))\subset\Omega(x)$, implying that $\Omega(x)$ is $\varphi$-invariant.

    \bigskip

2. Let $x\in\RC_C(\varphi)$, $U\in\VC_G$, and $\tau>0$ be arbitrary. Consider $V\in\VC_G$ satisfying 
$$V^{-1}V\subset U\hspace{.5cm}\mbox{ and }\hspace{.5cm} \psi_{\tau}(V)\subset V, \hspace{.5cm}\forall \tau\in\R,$$
where the first inclusion is assured by the continuity of the product and the inversion, and the second one by the fact that $\mathcal{Y}$ is elliptic.

Let us consider a $(V, \tau; \varphi)$-chain $\xi$ from $x$ to itself given by $\xi=\{n;x_0, \ldots, x_n, \tau_0, \ldots, \tau_{n-1}\}$. Defining, 
$$y_0=x \hspace{.5cm}\mbox{ and }\hspace{.5cm} y_i:= \psi_{\tau_0+\tau_1+\cdots+\tau_{i-1}}(x_i), \hspace{.5cm} i=1, \ldots, n,$$
we get that
$$y_{i+1}^{-1}(\varphi\circ\psi)_{\tau_i}(y_i)=\psi_{\sum_{j=0}^{i}\tau_j}(x_{i+1}^{-1})\varphi_{\tau_i}\left(\psi_{\tau_i}\left(\psi_{\sum_{j=0}^{i-1}\tau_j}(x_i)\right)\right)=\psi_{\sum_{j=0}^{i}\tau_j}(x_{i+1}^{-1})\varphi_{\tau_i}\left(\psi_{\sum_{j=0}^{i}\tau_j}(x_i)\right)$$
$$=\psi_{\sum_{j=0}^{i}\tau_j}(x_{i+1}^{-1})\psi_{\sum_{j=0}^{i}\tau_j}(\varphi_{\tau_i}(x_i))=\psi_{\sum_{j=0}^{i}\tau_j}\left(x_{i+1}^{-1}\varphi_{\tau_i}(x_i)\right)\in \psi_{\sum_{j=0}^{i}\tau_j}\left(V\right)\subset V,$$
showing that $\widehat{\xi}=\{n; y_0, \ldots, y_n, \tau_0, \ldots, \tau_{n-1}\}$ is a $(V, \tau; \varphi\circ\psi)$-chain from $x$ to $\psi_{\widehat{\tau}}(x)$, where $\widehat{\tau}=\tau(\xi)$ is the total time of $\xi$. Moreover, since $V$ is $\psi$-invariant, for any $\tau'\in\R$, the set $\psi_{\tau'}(\widehat{\xi}):=\{n; \psi_{\tau'}(y_0), \ldots, \psi_{\tau'}(y_n), \tau_0, \ldots, \tau_{n-1}\}$ is a $(V, \tau; \varphi\circ\psi)$-chain from $\psi_{\tau'}(x)$ to $\psi_{\tau'+\widehat{\tau}}(x)$. In particular, for any $p\geq 1$ the concatenation
$$
\widehat{\xi}\vee\psi_{\widehat{\tau}}(\widehat{\xi})\vee\cdots\vee\psi_{(p-1)\widehat{\tau}}(\widehat{\xi}),
$$
is a $(V, \tau; \varphi\circ\psi)$-chain from  $x$ for $\psi_{p\widehat{\tau}}(x)$.

On the other hand, the ellipticity of $\mathcal{Y}$ implies the existence of a subsequence satisfying
$$\psi_{n_k\widehat{\tau}}(x)\rightarrow \widehat{x}, \hspace{.5cm}\mbox{ for some }\hspace{.5cm}\widehat{x}\in G.$$ 
Hence, there exists $k_0\in\N$ such that $\psi_{n_k\widehat{\tau}}(x)\in \widehat{x}V$ for all $k>k_0$, and so
$$\psi_{n_{k_0}\widehat{\tau}}(x^{-1}\psi_{(n_k-n_{k_0})\widehat{\tau}}(x))= \psi_{n_{k_0}\widehat{\tau}}(x)^{-1}\psi_{n_k\widehat{\tau}}(x)\in (\widehat{x}V)^{-1}(\widehat{x}V)\subset V^{-1}V,$$
implying that
$$\forall k>k_0, \hspace{.5cm}x^{-1}\psi_{(n_k-n_{k_0})\widehat{\tau}}(x)\in\psi_{-n_{k_0}\widehat{\tau}}(V^{-1}V)= V^{-1}V\subset U.$$
Therefore, if $k\in\N$ is such that $p:=n_k-n_{k_0}-1>0$, the set 
$$\widehat{\xi}\vee\psi_{\widehat{\tau}}(\widehat{\xi})\vee\cdots\vee\psi_{p\widehat{\tau}}(\widehat{\xi})\vee \{\psi_{p\widehat{\tau}}(x), x, \widehat{\tau}\},$$
is a $(U, \tau; \varphi\circ\psi)$-chain from $x$ to $x$. Since $U$ and $\tau>0$ are arbitrary, we conclude that
$$x\in\RC_{C}(\varphi\circ\psi )\hspace{.5cm}\mbox{ and hence }\hspace{.5cm}\RC_{C}(\varphi)\subset \RC_{C}(\varphi\circ\psi).$$
Since the reverse flow $\psi^*_t:=\psi_{-t}$ is elliptical and satisfies
$$(\varphi\circ\psi)_t\circ\psi^*_s=\psi^*_s\circ(\varphi\circ\psi)_t, \hspace{.5cm}\forall t, s\in\R $$
the previous calculations allow us to conclude that
$$\RC_{C}(\varphi\circ\psi)\subset \RC_{C}((\varphi\circ\psi)\circ\psi^*)=\RC_{C}(\varphi),$$
concluding the proof.
\bigskip

3. Let $\{n; x_0, \ldots, x_n, \tau_0, \ldots, \tau_{n-1}\}$ be a $(U, \tau)$-chain from $x$ to $y$ with total time $T$.  Set $\tau_{-1}:=0$ and, for $i=0, \ldots, n$, consider the points 
$$g_i=\varphi_{\tau_0+\ldots+\tau_{i-1}}(g)x_i, \hspace{.5cm} i=1, \ldots, n.$$
Then, 
$$\forall i=0, \ldots, n-1, \hspace{.5cm}g_{i+1}^{-1}\varphi_{\tau_i}(g_i)=\left(\varphi_{\tau_0+\ldots+\tau_{i}}(g)x_{i+1}\right)^{-1}\varphi_{\tau_i}\left(\varphi_{\tau_0+\ldots+\tau_{i-1}}(g)x_i\right)$$
$$=\left(x_{i+1}^{-1}\varphi_{\tau_0+\ldots+\tau_{i}}(g)^{-1}\right)\varphi_{\tau_0+\ldots+\tau_{i}}(g)\varphi_{\tau_i}(x_i)=x_{i+1}^{-1}\varphi_{\tau_i}(x_i)\in U,$$
showing that $\xi=\{n; g_0, \ldots, g_n, \tau_0, \ldots, \tau_{n-1}\}$ is a $(U, \tau)$-chain from $gx$ to $\varphi_{T}(x)$ with total time $T$. Analogously, if we set 
$$g'_i:=\varphi_{-\tau_i-\ldots-\tau_{n-1}}(g)x_i, \hspace{.5cm}i=0, \ldots, m-1\hspace{.5cm}\mbox{ and }\hspace{.5cm}g'_m=gx,$$
we get that 
$$(g'_{i+1})^{-1}\varphi_{\tau_i}(g'_i)=\left(\varphi_{-\tau_{i+1}-\ldots-\tau_{n-1}}(g)x_{i+1}\right)^{-1}\varphi_{\tau_i}\left(\varphi_{-\tau_i-\ldots-\tau_{n-1}}(g)x_i\right)=x_{i+1}\varphi_{\tau_i}(x_i)\in U,$$
and hence, $\xi'=\{n; g'_0, \ldots, g'_n, \tau_0, \ldots, \tau_{n-1}\}$ is a $(U, \tau)$-chain from $\varphi_{-T}(g)x$ to $gy$ with total time $T$.


\bigskip

4. It is a direct consequence of the fact that
$$y\in\Omega(x)\hspace{.5cm}\iff\hspace{.5cm}x\in\Omega^*(y),$$
where $\Omega^*(y)$ is the chain limit set of the point $y\in G$, with respect to the flow $\varphi^*$.

In fact, let $U\in\mathcal{V}$ and $\tau>0$ and consider $V\in\mathcal{V}$ to be a neighborhood of the identity satisfying:
$$V=V^{-1}\hspace{.5cm}\mbox{ and }\hspace{.5cm} \varphi_t(V)\subset U, \hspace{.5cm}\forall t\in [- 2\tau, 0].$$
If $y\in\Omega(x)$, there exists a $(V, \tau)$-chain $\xi=\{n; x_0, \ldots, x_n, \tau_0, \ldots, \tau_{n-1}\}$ where $x_i\in G$, $\tau_i\geq \tau$, satisfy
$$x_0=x, \hspace{.5cm}x_n=y\hspace{.5cm}\mbox{ and }\hspace{.5cm}\forall i=0, \ldots, n-1, \hspace{.5cm}x_{i+1}^{-1}\varphi_{\tau_i}(x_i)\in V.$$  
By adding trivial jumps if necessary, allow us to assume that $\tau_i\in [\tau, 2\tau]$, and hence,
$$x_{i+1}^{-1}\varphi_{\tau_i}(x_i)\in V\hspace{.5cm}\iff \hspace{.5cm}\varphi_{\tau_i}\left(x_i^{-1}\varphi_{-\tau_i}(x_{i+1})\right)^{-1}\in V\hspace{.5cm}\implies\hspace{.5cm}x_i^{-1}\varphi^*_{\tau_i}(x_{i+1})\in \varphi_{-\tau_i}(V^{-1})\subset U.$$
Therefore, if $y_i:=x_{n-i}$, the set $\{n; y_0, \ldots, y_n, \tau_0, \ldots, \tau_{n-1}\}$ is an $(U, \tau)$-chain from $y$ to $x$ for the flow $\varphi^*$, showing that
$$y\in\Omega(x)\hspace{.5cm}\implies\hspace{.5cm}x\in\Omega^*(y).$$
In an analogous way, one shows the reverse implication, which proves the item.
\end{proof}

\subsection{Flows on homogeneous spaces}

A very useful tool in the study of the dynamics of flows of automorphisms is their projection to homogeneous spaces. Let $H\subset G$ be a closed subgroup and $\varphi$ a flow of automorphisms on $G$. By \cite[Proposition 4]{JPh1}, the flow $\varphi$ factors to a flow $\widehat{\varphi}$ on the homogeneous space $G/H$ if and only if $H$ is $\varphi$-invariant. In this case, $\widehat{\varphi}$ is completely determined by the relation
$$\forall t\in\R, \hspace{.5cm}\pi\circ\varphi_t=\widehat{\varphi}_t\circ\pi,$$
where $\pi:G\rightarrow G/H$ is the canonical projection.

The next result gives us a way to relate chains of $\varphi$ with the ones of $\widehat{\varphi}$. It will be essential in order to analyze the chain recurrent set of $\varphi$.

\begin{lemma}
\label{quotient}
    Let $H\subset G$ be a closed, $\varphi$-invariant Lie subgroup. It holds:

    \item[1.] If the pair $(H, G)$ satisfies
    \begin{equation}
        \label{homo0}
    \forall \varepsilon>0, \hspace{.3cm}\;\exists \,U\in \VC_G; \hspace{.5cm}\forall \,x\in G, \hspace{.5cm}\;\pi(xU)\subset B(\pi(x), \varepsilon),
        \end{equation}
    then $\pi(\RC_C(\varphi))\subset\RC_C(\widehat{\varphi})$
    
    \item[2.] If the pair $(H, G)$ satisfies
    \begin{equation}
        \label{homo}
    \forall U\in \VC_G, \hspace{.3cm}\;\exists \,\varepsilon>0; \hspace{.5cm}\forall \,x\in G, \hspace{.5cm}\;B(\pi(x), \varepsilon)\subset \pi(Ux),
        \end{equation}
then:
        \begin{itemize}
            \item[2.1.] If $\pi(y)\in\Omega(\pi(x), \widehat{\varphi})$, then for any $U\in\VC_G$, $\tau>0$ there exists $h\in H$ such that $hy^{-1}\in \Omega(x^{-1}, U, \tau, \varphi)$;

            \item[2.2.] If $\widehat{\varphi}$ is chain transitive and $H_0$ is a chain transitive subset, then $\varphi$ is chain transitive.
        \end{itemize}

\end{lemma}

\begin{proof} 1. Let $x\in\RC_C(\varphi)$ and $\tau, \varepsilon>0$. Let $U\in\VC_G$ satisfying property (\ref{homo0}) and consider an $(U, \tau)$-chain $\xi=\{n; x_0, \ldots, x_n, \tau_0, \dots, \tau_{n-1}\}$ is from $x$ to $x$. Then,  
$$x_{i+1}^{-1}\varphi_{\tau_i}(x_i)\in U\hspace{.5cm}\iff\hspace{.5cm}\varphi_{\tau_i}(x_i)\in x_{i+1}U\hspace{.5cm}\implies\hspace{.5cm}\widehat{\varphi}_{\tau_i}(\pi(x_i))=\varphi_{\tau_i}(x_i))\in\pi(x_{i+1}U)\subset B(\pi(x), \varepsilon),$$
showing that $\pi(\xi):=\{n; \pi(x_0), \ldots, \pi(x_n), \tau_0, \dots, \tau_{n-1}\}$ is a $(\varepsilon, \tau)$-chain of the induced flow $\widehat{\varphi}$ on $G/H$. Since $\varepsilon, \tau>0$ were arbitrary, we conclude that 
$x\in\RC_C(\widehat{\varphi})$, showing the result.

\bigskip

2.1. Let $U\in \VC_G$ and $\tau>0$, and consider $\varepsilon>0$ satisfying property (\ref{homo}) for $U$. Since $\pi(y)\in\Omega(\pi(x), \widehat{\varphi})$, there exists a $(\varepsilon, \tau; \widehat{\varphi})$-chain $\{n; \widehat{x}_0, \ldots, \widehat{x}_n;  \tau_0, \ldots, \tau_{n-1}\}$ from $\pi(x)$ to $\pi(y)$. In particular, 
$$\widehat{x}_{i+1}\in B(\widehat{\varphi}_{\tau_i}(\widehat{x}_i), \varepsilon)\hspace{.5cm}\implies\hspace{.5cm} \pi(x_{i+1})\in\pi(U\varphi_{\tau_i}(x_i))\hspace{.5cm}\implies\hspace{.5cm} x_{i+1}\in U\varphi_{\tau_i}(x_i)H,$$
where $\pi(x_i)=\widehat{x}_{i}\in G$. Let us define a $(U, \tau)$-chain as follows: $y_0:=x^{-1}$. For $i=0$, there exists by the previous an element $h_1\in H$ such that 
$$ x_1\in U\varphi_{\tau_0}(x)h_1^{-1}\hspace{.5cm}\iff\hspace{.5cm}
 (x_1h_1)\varphi_{\tau_0}(x^{-1})\in U,$$
and we define $y_1:=(x_1h_1)^{-1}$. Let $k\in\{2, \ldots, n-1\}$ and assume that we assured the existence of $h_i\in H$ such that
$$y_i=(x_ih_i)^{-1}\hspace{.5cm}\mbox{ satisfies }\hspace{.5cm} y_{i+1}^{-1}\varphi_{\tau_i}(y_i)\in U, \hspace{.5cm} i=1, \ldots k-1.$$
Let us define $y_{k+1}$ as follows: Since $x_{k+1}\in U\varphi_{\tau_{k}}(x_{k})H$, there exists $h'\in H$ such that
$$x_{k+1}h'\varphi_{\tau_{k}}(x_{k}^{-1})\in U\hspace{.5cm}\implies\hspace{.5cm}x_{k+1}h'\varphi_{\tau_k}(h_k)\varphi_{\tau_{k}}(y_{k})\in U,$$
and hence,
$$h_{k+1}:=h'\varphi_{\tau_k}(h_k)\in H, \hspace{.5cm}y_{k+1}:=(x_{k+1}h_{k+1})^{-1}\hspace{.4cm}\mbox{ satisfy }\hspace{.4cm} y_{k+1}^{-1}\varphi_{\tau_k}(y_{k})\in U,$$
shows that $\xi=\{n; y_0, \ldots, y_n, \tau_0, \ldots, \tau_{n-1}\}$ is an $(U, \tau)$-chain between $x^{-1}$ and $hy^{-1}$ for some $h\in H$, proving the item.

\bigskip

2.2. Let $U\in\VC_G$ and $\tau>0$. Since $\widehat{\varphi}$ is chain transitive, the previous item implies that
$$\forall x\in G, \;\pi(x^{-1})\in\Omega(\pi(e), \widehat{\varphi})\hspace{.5cm}\implies\hspace{.5cm} \forall x\in G, \;\exists h\in H, \;hx\in\Omega(e, U, \tau).$$
In particular, we have a well-defined function 
$$f_{U, \tau}=f:G\rightarrow \R, \hspace{.5cm}f(x):=\inf\{\varrho(h, e);\; h\in H, \; hx\in\Omega(e, U, \tau)\},$$
where $\varrho$ is a left-invariant metric. Since $U$ is open, $f$ is a continuous function.

Let us denote by $H_{\nu}$, $\nu\in I$ the connected components of $H$. Since $H$ is closed, the same happens with $H_{\nu}$ and hence, there exists  $h_{\nu}\in H_{\nu}$ such that 
$$\varrho(h_{\nu}, e)=\inf\{\varrho(h, e), h\in H_{\nu}\}.$$
Let us define the open sets $G_{\nu}:=\{x\in G, h_{\nu}x\in\Omega(e, U, \tau)\}$. Since any $h\in H$ can be written as $h=h_0h_{\nu}$ for some $h_0\in H_0$ and $\nu\in I$, item 3. of Proposition \ref{properties} gives us that
$$h_0h_{\nu}x=hx\in\Omega(e, U, \tau)\hspace{.5cm}\implies\hspace{.5cm}\exists T>0; \; h_{\nu}x\in \Omega(\varphi_{-T}(h_0^{-1}), U, \tau)\subset\Omega(e, U, \tau),$$
where the inclusion is due to the fact that $H_0$ is $\varphi$-invariant and chain transitive. 

As a consequence, 
$$G=\bigcup_{\nu\in I}G_{\nu}\hspace{.5cm} \mbox{ and }\hspace{.5cm} f(x)=\varrho(h_{\nu}, e), \hspace{.5cm}\forall x\in {G_{\nu}}.$$
However, the fact that $G$ is connected and that $G_{\nu}$ is open forces that $f(x)=f(e)=0$ for all $x\in G$, showing that $G=\Omega(e, U, \tau)$. By the arbitrariness of $U\in\VC_G$, $\tau>0$ we conclude that $G=\Omega(e)$.  An analogous analysis for $\varphi^*$ implies that $G=\Omega^*(e)$, which shows the transitivity of $\varphi$, concluding the result.
\end{proof}



\begin{remark}
The previous result shows that projecting chains only depends on left translations, while lifting chains back depends on right translations. Moreover, if $H$ is a normal subgroup, the canonical projection is a homomorphism, and hence condition (\ref{homo0}) follows directly. Condition (\ref{homo}) is also true for normal subgroups; however, its proof is not straightforward (see Proposition \ref{quotient}).

\end{remark}

\subsection{Uniform neighborhoods}

In this section we show that any chain recurrent point of a flow of automorphisms is contained in its central subgroup. We start with an adaptation of the concept of uniform neighborhood, which appears first in \cite{Hurley} for the context of discrete time dynamical systems (see also \cite{Hurley1}).

\begin{definition}
    Let $\AC, \mathcal{B}$ be nonempty subsets of $G$. We say that $\mathcal{B}$ is a right uniform neighborhood of $\AC$ if there exists $U\in\VC_G$ such that 
    $$\forall x\in \overline{\AC}, \hspace{.5cm} xU\subset \mathcal{B}\hspace{.5cm}.$$ 
\end{definition}

The next result assures the existence of natural uniform neighborhoods associated with a flow of automorphisms.

\begin{proposition}
\label{uniform}
    For any compact neighborhood $V\in\VC_{G^-}$, there exists $\tau_0>0$ such that the set $\AC_V^{+, 0}:=G^{+, 0}V$ is a right uniform neighborhood of $\varphi_{\tau}(\AC_V^{+, 0})$ for all $\tau\geq\tau_0$. 
\end{proposition}

\begin{proof} Let us consider $\delta>0$ such that $V\subset B^-(e, \delta)$. By Remark \ref{metricas}, there exists $\tau_0>0$ such that 
$$\forall\tau\geq\tau_0, \hspace{.5cm}\overline{\varphi_{\tau}(B^-(e, \delta))}\subset V\hspace{.5cm}\implies\hspace{.5cm}\hspace{.5cm}\forall\tau\geq\tau_0, \hspace{.5cm}\overline{\varphi_{\tau}(V)}\subset V.$$
Since $\AC_V^{+, 0}$ is open, for any $h\in \overline{\varphi_{\tau}(V)}$ there exists $U_h\in \VC_{G}$ such that $hU^2_h\subset \AC_V^{+, 0}$. By the compactness of $\overline{\varphi_{\tau}(V)}$, the open cover 
    $\{hU_h\}_{h\in \overline{\varphi_{\tau}(V)}}$  admits a finite subcover $\{h_1U_{h_1}, \ldots, h_kU_{h_k}\}$, and so, $U:=\bigcap_{j=1}^k U_{h_j}\in \VC_{G}$.

On the other hand, by the compactness of $\overline{\varphi_{\tau}(V)}$ and the fact that $G^{+, 0}$ is closed and $\varphi$-invariant, it holds that 
$$\overline{\varphi_{\tau}(\AC^{+, 0}_V)}=\overline{\varphi_{\tau}(G^{+, 0}V)}=\overline{G^{+, 0}\varphi_{\tau}(V)}=G^{+, 0}\overline{\varphi_{\tau}(V)}.$$
Therefore, for any $x\in \overline{\varphi_{\tau}(\AC^{+, 0}_V)}$, there exists $g\in G^{+, 0}$ and $h\in \overline{\varphi_{\tau}(V)}$ such that $x=gh$. By the previous, $h\in h_iU_{h_i}$ for some $i\in\{1, \ldots, k\}$, implying that 
$$xU=ghU\subset gh_iU_{h_i}U\subset gh_iU_{h_i}^2\subset g\AC_V^{+, 0}\subset \AC_V^{+, 0},$$
showing the result.
\end{proof}

\bigskip

An important feature of uniform neighborhoods is the fact that chains starting inside them remain inside. This property implies the following result: 

\begin{theorem}
\label{teo}
    It holds that
    $$\RC_C(\varphi)\cap G^{+, 0}G^-=\RC_C\left(\varphi|_{G^0}\right).$$
\end{theorem}

\begin{proof} For any $x\in \RC_C(\varphi)\cap G^{+, 0}G^-$, there exists a compact neighborhood $V\in\VC_{G^-}$ such that $x\in \AC_V^{+, 0}$. By the previous result, there exists $\tau_0>0$ such that, for any $\tau\geq\tau_0$, the set $\AC_V^{+, 0}$ is a right uniform neighborhood of $\overline{\varphi_{\tau}(\AC_V^{+, 0})}$. In particular, there exists by the previous result $U\in\VC_G$ such that $zU\in\AC_V^{+, 0}$ for any $z\in \overline{\varphi_{\tau}(\AC_V^{+, 0})}$. Let then $W\in\VC_G$ with $W\subset U$ and consider  a $(W^{-1}, 2\tau)$-chain  $\xi:=\{n; x_0, \ldots, x_n, \tau_0, \ldots, \tau_{n-1}\}$ from $x$ to itself. Since
$$
\forall i=0, \ldots, n-1, \hspace{.5cm} x_{i+1}^{-1}\varphi_{\tau_i}(x_i)\in W^{-1}\hspace{.5cm}\iff\hspace{.5cm} x_{i+1}\in\varphi_{\tau_i}(x_i)W,
$$
if $x_i\in\AC_V^{+, 0}$, then 
$$
\varphi_{\tau_i}(x_i)\in\varphi_{\tau_i}(\AC_V^{+, 0})=\varphi_{\tau}\left(\varphi_{\tau_i-\tau}(\AC_V^{+, 0})\right)\subset\varphi_{\tau}(\AC_V^{+, 0})\hspace{.5cm}\implies\hspace{.5cm}x_{i+1}\in\AC_V^{+, 0},
$$
where for the last inclusion we used that $\tau_i\geq 2\tau$ implies that $\varphi_{\tau_i-\tau}(\AC_V^{+, 0})\subset \AC_V^{+, 0}$. Since $x_0=x\in\AC_V^{+, 0}$, we conclude that all the points $x_i$, $i=1, \ldots, n$ of the chain belong to $\AC_V^{+, 0}$. Moreover, it also holds that 
$$x=x_n\in\varphi_{\tau_{n-1}}(x_{n-1})W\subset\varphi_{\tau}(\AC_V^{+, 0})W.$$
Since $W\in\VC_G$ satisfying $W\subset U$ was arbitrary, we conclude that 
$$x\in \overline{\varphi_{\tau}(\AC_V^{+, 0})}.$$
On the other hand, the fact that  $\bigcap_{\tau>0}\varphi_{\tau}(V)=\{e\}$ allows us to obtain
$$x\in \bigcap_{\tau>0}\overline{\varphi_{\tau}(\AC_V^{+, 0})}=\bigcap_{\tau>0}G^{+, 0}\overline{\varphi_{\tau}(V)}\subset G^{+, 0}\hspace{.5cm}\implies\hspace{.5cm}\RC_C(\varphi) \cap G^{+, 0}G^-\subset \RC_C(\varphi)\cap G^{+, 0}.$$

Let us show that in fact
$$\RC_C(\varphi)\cap G^{+, 0}\subset\RC_C\left(\varphi|_{G^{+, 0}}\right).$$

Let $\widehat{U}\in\VC_{G}$ and consider a compact neighborhood $V\in V_{G^-}$ such that
$\overline{VV^{-1}}\subset \widehat{U}$. By Proposition \ref{uniform}, there exists $\tau_0>0$ and $U\in\VC_G$ such that 
$$\forall \tau\geq\tau_0, \hspace{.5cm}\overline{\varphi_{\tau}(V)}\subset V\hspace{.5cm}\mbox{ and }\hspace{.5cm}\overline{\varphi_{\tau}(\AC_V^{+, 0})}U\subset\AC_V^{+, 0}.$$
Consider also $W\in\VC_G$ satisfying 
$$W\subset U\hspace{.5cm}\mbox{ and }\hspace{.5cm}VW^{-1}V^{-1}\subset \widehat{U},$$
which exists by the compactness of $\overline{VV^{-1}}$. Let then $x\in \RC_C(\varphi)\cap G^{+, 0}$. Since $x\in\AC_V^{+, 0}$ for any $V\in\VC_{G^-}$, the previous calculations implies that a $(W^{-1}, 2\tau)$-chain $\xi:=\{n; x_0, \ldots, x_n, \tau_0, \ldots, \tau_{n-1}\}$ from $x$ to itself satisfies that $x_i\in\AC_{V}^{+, 0}$, for $i=1, \ldots n$. Write $$x_i=g_iv_i, \hspace{.5cm}\mbox{ where }\hspace{.5cm} g_i\in G^{+, 0}\hspace{.5cm}\mbox{ and }\hspace{.5cm} v_i\in V\subset G^-,$$
with $v_0=v_n=e$ and $g_0=g_n=x$. It holds that
$$x_{i+1}^{-1}\varphi_{\tau_i}(x_i)\in W^{-1}\hspace{.5cm}\iff \hspace{.5cm}v_{i+1}^{-1}g_{i+1}^{-1}\varphi_{\tau_i}(g_i)\varphi_{\tau_i}(v_i)\in W^{-1}$$
$$\implies\hspace{.5cm} G^{+, 0}\ni g_{i+1}^{-1}\varphi_{\tau_i}(g_i)\in v_{i+1}W^{-1}\varphi_{\tau_i}(v_i)^{-1}\subset VW^{-1}V^{-1}\subset \widehat{U}.$$
Therefore, $\widehat{\xi}=\{n; g_0, g_1, \ldots, g_n, \tau_0, \ldots, \tau_{n-1}\}$ is a $(\widehat{U}\cap G^{+, 0}, 2\tau)$-chain from $x$ to itself for the flow $\varphi|_{G^{+, 0}}$. Since any element of $\VC_{G^{+, 0}}$ contains a neighborhood of the form $\widehat{U}\cap G^{+, 0}$ for some $\widehat{U}\in\VC_{G}$, the arbitrariness of former in the previous calculations assures that $x\in\RC_{C}\left(\varphi|_{G^{+, 0}}\right)$, showing that 
$$\RC_{C}\left(\varphi\right)\cap G^{+, 0}\subset \RC_{C}\left(\varphi|_{G^{+, 0}}\right),$$
as desired. By doing the same calculations as previously for the reverse flow $\left(\varphi|_{G^{+, 0}}\right)^*=\varphi^*|_{G^{+, 0}},$
allows us to conclude that
$$\RC_{C}\left(\varphi|_{G^{+, 0}}\right)=\RC_{C}\left(\varphi^*|_{G^{+, 0}}\right)=\RC_{C}\left(\varphi^*|_{G^{+, 0}}\right)\cap G^0G^+\subset \RC_{C}\left(\left(\varphi^*|_{G^{+, 0}}\right)|_{G^0}\right)=\RC_{C}\left(\varphi^*|_{G^0}\right)=\RC_{C}\left(\varphi|_{G^0}\right),$$
showing that 
$$\RC_C(\varphi)\cap G^{+, 0}G^-\subset \RC\left(\varphi|_{G^0}\right).$$
Since the opposite inclusion always holds, the result is proved.
\end{proof}

\bigskip

As a direct consequence, we have the following:

\begin{corollary}
\label{cor}
Let $\varphi$ be a flow of automorphisms on a Lie group $G$. If $G$ is decomposable by $\varphi$, then $$\RC_C(\varphi)=\RC_C\left(\varphi|_{G^0}\right).$$
\end{corollary}

\begin{proof}
    In fact, if $\varphi$ decomposes $G$, then $G=G^{+, 0}G^-$ which by Theorem \ref{teo}  implies that 
    $$\RC_C\left(\varphi|_{G^0}\right)\subset\RC_C(\varphi)=\RC_C(\varphi)\cap G^{+, 0}G^-=\RC_C\left(\varphi|_{G^0}\right).$$
\end{proof}

\section{The chain recurrent set of flows of automorphisms of decomposable Lie groups}

In this section, we prove our main result concerning the chain recurrent set of flow of automorphisms, namely, we prove the following:

\begin{theorem}
\label{main}
    Let $G$ be a connected Lie group and $\varphi$ a flow of automorphisms on $G$. If $G$ is decomposable by $\varphi$, $$\RC_C(\varphi)=\RC_C\left(\varphi|_{G^0}\right)=G^0,$$
    where $G^0$ is the central subgroup associated with $\varphi$.
\end{theorem}

We prove the previous theorem in the next sections, where we split the proof depending on the class to which the group $G$ belongs. Before starting, let us make some reductions in the case we have to consider. By Proposition \ref{properties}, the chain recurrent set of a flow of automorphisms satisfies
$$\RC_{C}(\varphi)=\RC_C\left(\varphi^{\HC, \NC}\right),$$
where $\varphi^{\HC}$ and $\varphi^{\NC}$ are the hyperbolic and nilpotent parts of $\varphi$, respectively. Moreover, if $\varphi$ is a flow of automorphisms that decomposes $G$, then $\varphi^{\HC, \NC}$ also decomposes $G$, and by Theorem \ref{teo} and its corollary, we get that
$$\RC_C\left(\varphi^{\HC, \NC}\right)=\RC_C\left(\varphi^{\HC, \NC}|_{G^0}\right)=\RC_C\left(\varphi^{\NC}|_{G^0}\right),$$
where for the last equality we used that 
$$\varphi^{\HC, \NC}=\varphi^{\HC}\circ\varphi^{\NC}=\varphi^{\NC}\circ\varphi^{\HC}\hspace{.5cm}\mbox{ and }\hspace{.5cm}\varphi^{\HC}|_{G^0}=\id_{G^0}.$$
Since, in the decomposable case, $G^0$ is connected, we only have to show that nilpotent flows of automorphisms are chain transitive.

Another assumption that we can make w.l.o.g. is that the group $G$ is simply connected. In fact, since any flow of automorphisms $\varphi$ on a connected group $G$ can be lifted to a flow of automorphisms $\widetilde{\varphi}$ on the simply connected covering $\widetilde{G}$ of $G$. The fact that
$$\forall t\in\R, \hspace{.5cm}\pi\circ\widetilde{\varphi}_t=\varphi_t\circ\pi,$$
implies in particular that $\pi(\widetilde{G}^0)=G^0$. Since $\pi\left(\RC_C(\widetilde{\varphi})\right)\subset \RC_C(\varphi)$ (Lemma \ref{quotient}), we have that
$$\widetilde{G}^0\subset \RC_C(\widetilde{\varphi})\hspace{.5cm}\implies\hspace{.5cm} G^0\subset \RC_C(\varphi).$$

\subsection{Nilpotent flow of automorphisms on solvable Lie groups}

In this section, we show that a nilpotent flow of automorphisms on a solvable Lie group is chain transitive.

\begin{proposition}
\label{solvable}
Any nilpotent flow of automorphisms on a connected solvable Lie group is chain transitive.
\end{proposition}

\begin{proof} In what follows we show that $G$ is chain transitive by induction on its dimension. As commented in the beginning of the section, we assume w.l.o.g. that $G$ is simply connected.

The one-dimensional case follows as a consequence of the abelian case proved in \cite[Theorem 3.3]{FCASEV}. Let us then assume that the result is true for any nilpotent flow of automorphisms on any solvable Lie group with dimension smaller or equal to $k$ and let $G$ be a simply connected solvable Lie group with $\dim G=k+1$ and $\{\varphi_{\tau}\}_{\tau\in\R}$ a nilpotent flow of automorphisms on $G$.

Since $G$ is solvable, its Lie algebra $\fg$ is also solvable, and hence, its derived series satisfies
$$\fg^0:=\fg\supset \fg^{(1)}\supset\cdots\supset\fg^{(p)}\supset\fg^{(p+1)}=\{0\}, \hspace{.5cm}\mbox{ where }\hspace{.5cm}\fg^{(i)}:=[\fg^{i-1}, \fg^{i-1}], \hspace{.5cm}\mbox{ for }\hspace{.5cm}i\geq 1.$$
Let us denote by $H$ the connected Lie group with Lie algebra $\fg^{(p)}$. The subgroup $H$ is a closed, normal, and abelian Lie group. Moreover, the fact that $\fg^{(i)}$ is invariant by derivations implies that $H$ is $\varphi$-invariant. By the abelian case, it holds that
$$H=\RC_C(\varphi|_H)\subset \RC_C(\varphi).$$
On the other hand, the $\varphi$-invariance of $H$ guarantees the existence of an induced flow of automorphisms $\widehat{\varphi}$ on $G/H$ satisfying $$\forall t\in \R, \hspace{.5cm}\pi\circ\varphi_t=\widehat{\varphi}_t\circ\pi.$$
 Since $\widehat{\varphi}$ is a nilpotent flow of automorphisms on $G/H$, the inductive hypothesis implies that $\RC_C(\widehat{\varphi})=G/H$. By Lemma \ref{quotient}, we conclude that $G\subset \RC_C(\varphi)$, proving the result.
\end{proof}

\subsection{Nilpotent flow of automorphisms on semisimple Lie groups}

In this section, we consider the semisimple case. We separate the analysis into the compact and the noncompact semisimple Lie groups.

For both cases, the fact that the Lie algebra $\fg$ of $G$ is semisimple, implies that $\DC=\ad(X)$ for some $X\in\fg$. Moreover, the flow of 
 automorphisms is given as $\varphi_t=C_{\rme^{tX}}$.

\subsubsection{Compact semisimple Lie groups}

Since $G$ is a compact semisimple group, the Cartan-Killing form 
$$(X, Y)\in\fg\times\fg\mapsto\tr(\ad(X)\ad(Y)),$$
is a non-degenerated and negative-definite bilinear form on $\fg$. In particular, 
$$\langle X, Y\rangle=-\tr(\ad(X)\ad(Y)),$$
is an inner product on $\fg$. Since it is invariant by automorphisms, we get that $\varphi_t=C_{\rme^{tX}}$ is an elliptical flow of automorphisms. By item 2. of Proposition \ref{properties} we conclude that 
$$G=\RC_C\left(\id\right)=\RC_C\left(\id\circ\varphi\right)=\RC_G\left(\varphi\right),$$
where $\id:G\rightarrow G$ is the identity of $G$. In particular, we have the following.

\begin{proposition}
    \label{solvable}
 Any flow of automorphisms on a connected and compact semisimple Lie group is elliptical and hence chain transitive.
\end{proposition}

\subsubsection{Noncompact semisimple Lie groups}

Here we consider noncompact, connected semisimple Lie groups.

\begin{proposition}
    \label{semisimple}
    Any nilpotent flow of automorphisms on a noncompact, connected semisimple Lie group is chain transitive.
\end{proposition}

\begin{proof}
Let $G$ be a noncompact, connected semisimple Lie group with Lie algebra $\fg$ and $\varphi$ a flow of automorphisms on $G$. As previously, there exists $X\in\fg$ such that $\varphi_t=C_{\rme^{tX}}$. Under the assumption that $\varphi$ is nilpotent, we get that $X$ is a nilpotent element. Therefore, there exists a minimal parabolic subalgebra $\mathfrak{p}$ with $X\in\mathfrak{p}$ (see \cite[Theorem 7.2]{hell}). If $P$ is the associated minimal parabolic subgroup of $\fp$, then 
$$\rme^{tX}\in P_0\hspace{.5cm}\implies\hspace{.5cm}P \mbox{ is }\varphi\mbox{-invariant}.$$ 
In particular, we can consider the induced flow $\widetilde{\varphi}$ on the flag manifold $G/P$. Since,
$$\widetilde{\varphi}_t(xP)=\rme^{tX}g\rme^{-tX}P=\rme^{tX}xP,$$
we have that $\widetilde{\varphi}$ coincides with the flow induced by one parameter subgroup $\{\rme^{tX}, t\in\R\}$, which by  \cite[Theorem 4.2]{Ferraiol} is chain transitive.  On the other hand, the fact that $P$ is a solvable subgroup implies by Proposition \ref{solvable} that $P_0$ is a chain transitive subset.

Since the flag manifold $G/P$ is compact, relation (\ref{homo}) is satisfied (see Proposition \ref{property} (i)), and hence, by Lemma \ref{quotient}, we conclude that $\varphi$ is chain transitive, proving the result.
\end{proof}




\subsection{The general case}

In this section, we conclude the proof of our main result by showing that nilpotent flows of automorphisms on connected Lie groups are chain transitive.

\begin{proposition}
Any nilpotent flow of automorphisms on a connected Lie group is chain transitive.
\end{proposition}

\begin{proof} Let $\varphi$ be a nilpotent flow of automorphisms on $G$ and denote by $R$ the solvable radical of $G$. Since $R$ is normal and $\varphi$-invariant, the induced flow $\widehat{\varphi}$ on $G/R$ determined by the relation
    $$\forall t\in\R, \hspace{.5cm}\pi\circ\widehat{\varphi}_t=\varphi_t\circ\pi,$$
is a nilpotent flow of automorphisms. Since $G/R$ is a connected semisimple Lie group, the previous section assures that $\widehat{\varphi}$ is chain transitive. On the other hand, $R$ is solvable, and the restriction $\varphi|_R$ is nilpotent, implying by Proposition \ref{solvable} that $\varphi|_R$ is chain transitive, that is, $R$ is a chain transitive subset. To conclude, the fact that $R$ is normal implies by Proposition \ref{properties} that $R$ and $G$ satisfy property (\ref{homo}), and hence, $\varphi$ is chain transitive, as desired.
\end{proof}

\begin{remark}
    Let us note that Theorem \ref{main} implies that flow of automorphisms on connected Lie groups satisfy the restriction property, that is,
    $$\RC_C(\varphi)=\RC_C\left(\varphi|_{\RC_C(\varphi)}\right).$$
\end{remark}

\appendix

\section{Homogeneous spaces}

Let $(H, G)$ be a pair of connected Lie groups where $H\subset G$ is a closed subgroup. In this section we are interested in the pairs satisfying the property (\ref{homo}), that is, pairs with the property that
$$\forall U\in\VC_G, \;\exists\varepsilon>0; \hspace{.5cm}\forall x\in G, \hspace{.5cm}B(\pi(x), \varepsilon)\subset \pi(Ux),$$
where $\pi:G\rightarrow G/H$ is the canonical projection and the metric considered in the homogeneous space is compatible with the topology of $G/H$.

\begin{proposition}
\label{property}
    The following pair of groups satisfy property (\ref{homo}):

    \begin{itemize}
        \item[1.] $(H, G)$ with $H$ normal on $G$; 
    
        \item[2.] $(H, G)$ such that $G/H$ is compact;

           \end{itemize}

\end{proposition}

\begin{proof} 1. If $H$ is a normal subgroup, the homogeneous space $G/H$ is a Lie group and we can consider a right-invariant metric on $G/H$. Since, the canonical projection $\pi:G\rightarrow G/H$ is an open map, for any $U\in\VC_G$, there exists $\varepsilon>0$ such that $B(\pi(e), \varepsilon)\subset \pi(U)$. Hence, 
$$\pi^{-1}(B(\pi(e), \varepsilon))\subset UH\hspace{.5cm}\implies\hspace{.5cm}\forall x\in G, \hspace{.5cm}\pi^{-1}(B(\pi(e), \varepsilon))x\subset UHx=UxH,$$
where for the last equality we used that $H$ is a normal subgroup. On the other hand, the fact that $\pi$ is a surjective homomorphism and the right-invariance of the metric on $G/H$ implies that
$$\pi\left(\pi^{-1}(B(\pi(e), \varepsilon))x\right)=\pi\left(\pi^{-1}(B(\pi(e), \varepsilon))\right)\pi(x)=B(\pi(e), \varepsilon)\pi(x)=B(\pi(x), \varepsilon
),$$
implying that 
$$\forall x\in G, \hspace{.5cm}B(\pi(x), \varepsilon
)\subset \pi(UxH)=\pi(Ux),$$
as desired.

\bigskip

2. Following \cite[Lemma 11.7]{SM2}, the compactness of $G/H$ implies the existence of a compact $C\subset G$ with $e\in\inner C$ and such that $G=CH$. Let $U\in \VC_{G}$ and consider $c\in C$. By continuity, there exist $V_c\in \VC_G$ and $W_c$ a neighborhood of $c$ such that 
\begin{equation}
    \label{eq1}
\forall x\in W_c, \hspace{.5cm}xV_cx^{-1}\subset U.
\end{equation}
Since $C\subset\cup_{c\in C}W_c$, the compactness of $C$ implies the existence of $c_1, \ldots, c_n\in C$ such that 
$C\subset \cup_{i=1}^nW_{c_i}$. Let $V_{c_i}\in \VC_G$ the neighborhood that, together with $W_{c_i}$, satisfy (\ref{eq1}), and define $V:=\cap_{i=1}^n V_{c_i}$. 

Then, for any $x\in C$ there exists $i\in\{1, \ldots, n\}$ such that $x\in W_{c_i}$. Therefore,
$$xVx^{-1}\subset xV_{c_i}x^{-1}\subset U.$$
Let us assume that the assertion does not holds. Hence, there exists $U\in\VC_G$, $x_n, y_n\in G$ and $\varepsilon_n\rightarrow 0$ such that 
$$\pi(y_n)\in B(\pi(x_n), \varepsilon_n)\hspace{.5cm}\mbox{ and }\hspace{.5cm}\pi(y_n)\notin \pi(Ux_n).$$

Since $\pi(Hg)=\pi(g)$ for any $g\in G$, we can assume w.l.o.g. that $x_n, y_n\in C$. By considering subsequences if necessary, we can assume that $x_n\rightarrow x$ and $y_n\rightarrow y$. Then, 
$$\pi(y_n)\in B(\pi(x_n), \varepsilon_n)\hspace{.5cm}\implies\hspace{.5cm} \pi(x)=\pi(y)\hspace{.5cm}\iff\hspace{.5cm}x^{-1}y\in H.$$
On the other hand, if $V\in \VC_G$ satisfies $xVx^{-1}\subset U$ for all $x\in C$, then 
$$\pi(y_n)\notin \pi(x_nU)\hspace{.5cm}\iff\hspace{.5cm} \pi(x_n^{-1}y_n)\notin \pi(x_n^{-1}Ux_n)\hspace{.5cm}\implies\hspace{.5cm} \pi(x_n^{-1}y_n)\notin \pi(V).$$
Since $\pi(V)$ is a neighborhood of $\pi(e)$ in $G/H$ we conclude that  
$$\pi(x_n^{-1}y_n)\rightarrow \pi(x^{-1}y)\notin \inner\pi(V).$$
However, $x^{-1}y\in H$ implies that $\pi(x^{-1}y)=\pi(e)\in\inner\pi(V)$, which is a contradiction.
\bigskip


\end{proof}


    

\end{document}